\documentclass[12pt]{amsart}

\usepackage{amssymb}
\usepackage{amscd}

\title{Algebraic curves admitting inner and outer Galois points} 
\author{Satoru Fukasawa}

\address{Department of Mathematical Sciences, Faculty of Science, Yamagata University, Kojirakawa-machi 1-4-12, Yamagata 990-8560, Japan} 
\email{s.fukasawa@sci.kj.yamagata-u.ac.jp}

\thanks{The author was partially supported by JSPS KAKENHI Grant Number JP19K03438.}

\newtheorem{theorem}{Theorem}
\newtheorem{proposition}{Proposition}

\newtheorem{lemma}{Lemma}

\theoremstyle{definition}
\newtheorem{remark}{Remark}

\begin{document}
\begin{abstract}
There are two purposes in this article.  
One is to present a criterion for the existence of a birational embedding into a projective plane with inner and outer Galois points for algebraic curves. 
Another is to classify plane curves of degree $d$ admitting an inner Galois point $P$ and an outer Galois point $Q$ with $G_{P}G_{Q}=G_P \rtimes G_Q$ or $G_P \ltimes G_Q$, under the assumption that the characteristic $p$ is zero or $p$ does not divide $d-1$. 
\end{abstract}

\maketitle 

\section{Introduction} 
In 1996, Hisao Yoshihara introduced the notion of a Galois point for a plane curve $C \subset \mathbb{P}^2$ over an algebraically closed field $k$ of characteristic $p \ge 0$ (\cite{fukasawa1, miura-yoshihara, yoshihara}): a point $P \in \mathbb{P}^2$ is called a Galois point, if the extension $k(C)/\pi_P^*k(\mathbb{P}^1)$ of function fields induced by the projection $\pi_P$ from $P$ is a Galois extension. 
The associated Galois group is denoted by $G_P$. 
Furthermore, a Galois point $P$ is said to be inner (resp. outer), if $P \in C \setminus {\rm Sing}(C)$ (resp. if $P \in \mathbb{P}^2 \setminus C$). 
In this article, we investigate (possibly singular) plane curves $C$ in the case where there exist an inner Galois point $P$ and an outer Galois point $Q$. 

Since we do not assume that the smoothness of $C$, sometimes we start with a (reduced, irreducible) smooth projective curve $X$. 
In this case, $C$ is the image of some morphism $\varphi: X \rightarrow \mathbb{P}^2$, which is birational onto its image. 

In \cite{fukasawa3}, the present author presented a criterion for the existence of a birational embedding with two Galois points $P$ and $Q$. 
In this criterion, it is assumed that both points $P$ and $Q$ are inner, or outer.   
In this article, we consider the case where $P$ is inner and $Q$ is outer.

\begin{theorem} \label{main1} 
Let $G_1$ and $G_2$ be finite subgroups of ${\rm Aut}(X)$, and let $P \in X$. 
Then the two conditions
\begin{itemize}
\item[(a)] $X/{G_1} \cong \Bbb P^1$ and $X/G_2 \cong \mathbb{P}^1$, and    
\item[(b)] there exists $\eta \in G_2$ such that $P+\sum_{\sigma \in G_1} \sigma (\eta(P))=\sum_{\tau \in G_2} \tau (P)$
\end{itemize}
are satisfied, if and only if there exists a birational embedding $\varphi: X \rightarrow \mathbb P^2$ of degree $|G_2|$ and an outer Galois point $Q \in \mathbb{P}^2\setminus \varphi(X)$ exists for $\varphi(X)$ such that $\varphi(P)$ is an inner Galois point, and $G_{\varphi(P)}=G_1$ and $G_Q=G_2$. 
\end{theorem}

In \cite{fukasawa-speziali}, the present author and Speziali classified plane curves admitting two outer Galois points $P_1$ and $P_2$ with $G_{P_1}G_{P_2}=G_{P_1} \rtimes G_{P_2}$. 
The second purpose in this article is to classify $C$ admitting an inner Galois point $P$ and an outer Galois point $Q$ with $G_PG_Q=G_P \rtimes G_Q$ or $G_P \ltimes G_Q$. 
The case where $G_PG_Q=G_P \rtimes G_Q$ is determined as follows. 

\begin{theorem} \label{main2}
Let $C$ be a plane curve of degree $d \ge 3$. 
Assume that $p=0$ or $d-1$ is prime to $p$. 
Then the following conditions are equivalent. 
\begin{itemize}
\item[(a)] There exist an inner Galois point $P$ and an outer Galois point $Q$ such that $G_PG_Q=G_P \rtimes G_Q$; 
\item[(b)] There exist an inner Galois point $P$ and an outer Galois point $Q$ such that $G_PG_Q=G_P\times G_Q$; 
\item[(c)] $C$ is in one of the following cases: 
\begin{itemize}
\item{} $p=0$ or $d$ is prime to $p$, and $C$ is projectively equivalent to the plane curve defined by $x^{d-1}+y^d+c=0$, where $c=0$ or $1$; 
\item{} $p>0$, $d$ is divisible by $p$, and $C$ is projectively equivalent to the plane curve defined by $x^{d-1}+(\sum_{m \ | \ p^i-1}\alpha_i y^{p^i})^m+c=0$, where $d=p^em$, $m$ is not divisible by $p$, $\alpha_e\alpha_0 \ne 0$, and $c=0$ or $1$.  
\end{itemize}
\end{itemize}
\end{theorem}

On the other hand, there exist examples of plane curves such that $G_PG_Q=G_P\ltimes G_Q$ and $G_PG_Q \ne G_P\times G_Q$. 
In this case, such curves are classified as follows. 

\begin{theorem}\label{main3}
Let $C$ be a plane curve of degree $d \ge 3$. 
Assume that $p=0$ or $d(d-1)$ is prime to $p$. 
Then there exist an inner Galois point $P$ and an outer Galois point $Q$ such that $G_PG_Q=G_P \ltimes G_Q$ and $G_PG_Q \ne G_P \times G_Q$, if and only if $C$ is in one of the following cases: 
\begin{itemize}
\item[(a)] $p \ne 2, 3$, $d=3$ and $C$ is projectively equivalent to the curve defined by 
$$y^2x+(x+1)^2(x-8)=0.$$
\item[(b)] $p \ne 2, 3$, $d=4$ and $C$ is projectively equivalent to the curve defined by 
$$y^3x+(x+1)^3(x+9)=0. $$ 
\end{itemize}
\end{theorem} 

For the proof of the if-part of Theorem \ref{main3}, we prove the existence of triples $(G_1, G_2, P)$ with conditions (a) and (b) in Theorem \ref{main1} such that $G_1G_2 \cong S_3$ or $A_4$, where $S_3$ (resp. $A_4$) is the symmetric (resp. alternative) group of degree three (resp. of degree four).  
For the case where $d$ is divisible by $p$, the following holds. 

\begin{proposition} \label{main4}
Let $p>0$ and let $d$ be divisible by $p$. 
Then there exist an inner Galois point $P$ and an outer Galois point $Q$ such that $G_PG_Q=G_P \ltimes G_Q$, $G_PG_Q \ne G_P \times G_Q$ and $\overline{PQ}$ is not the tangent line at $P$, if and only if  $d$ is a power of $p$, and $C$ is projectively equivalent to the curve defined by 
$$ y^{d-1}x+(x+1)^d=0. $$
\end{proposition}

\section{Proof of Theorem \ref{main1}} 

\begin{proof}[Proof of Theorem \ref{main1}]
We consider the if-part. 
Since $\varphi(P)$ and $Q$ are Galois points with $G_{\varphi(P)}=G_1$ and $G_Q=G_2$, it follows that fixed fields $k(X)^{G_{\varphi(P)}}=k(X)^{G_1}$ and $k(X)^{G_Q}=k(X)^{G_2}$ are rational. 
Condition (a) is satisfied. 
Let $\ell$ be the line passing through points $\varphi(P)$ and $Q$, and let $D=\varphi^*\ell$. 
Since $Q$ is an outer Galois point and $G_Q=G_2$, 
$$ D=(\pi_Q \circ \varphi)^*(\pi_Q(\varphi(P)))=\sum_{\tau \in G_2}\tau(P) $$
(see \cite[III.7.1, III.7.2, III.8.2]{stichtenoth}).
Let $\eta \in G_2\setminus\{1\}$. 
Then $\varphi(\eta(P)) \in \ell$. 
To prove that condition (b) is satisfied, we would like to prove that 
$$ D-P=\sum_{\sigma \in G_1}\sigma(\eta(P)). $$
Assume that $\eta(P) \ne P$.  
Then $\varphi(\eta(P)) \in \ell \setminus \{\varphi(P)\}$. 
Since $\varphi(P)$ is an inner Galois point and $G_{\varphi(P)}=G_1$, it follows that  
$$ D-P=(\pi_{\varphi(P)}\circ \varphi)^*(\pi_{\varphi(P)} \circ \varphi(\eta(P)))=\sum_{\sigma \in G_1}\sigma(\eta(P)).$$
The claim follows. 
Assume that $\eta(P)=P$. 
Then $\ell$ is the tangent line at $\varphi(P)$.  
Since $\varphi(P)$ is an inner Galois point, it follows that $P \in {\rm supp} (\pi_{\varphi(P)} \circ \varphi)^*(\pi_{\varphi(P)}\circ\varphi(P))$ and  
$$ D-P=(\pi_{\varphi(P)}\circ\varphi)^*(\pi_{\varphi(P)}\circ\varphi(P))=\sum_{\sigma \in G_1}\sigma(P).$$
It follows from the condition $P=\eta(P)$ that the claim follows.

We consider the only-if part. 
Let $\eta \in G_2$ be an automorphism as in condition (b). 
By condition (a), there exists a function $f \in k(X)$ such that 
$$ k(X)^{G_1}=k(f), \ (f)_{\infty}=\sum_{\sigma \in G_1}\sigma(\eta(P)) $$
(see \cite[III.7.1, III.7.2, III.8.2]{stichtenoth}). 
Similarly, there exists $g \in k(X)$ such that 
$$ k(X)^{G_2}=k(g), \ (g)_{\infty}=\sum_{\tau \in G_2}\tau(P). $$
Considering condition (b), we take a divisor 
$$ D:=P+\sum_{\sigma \in G_1}\sigma(\eta(P))=\sum_{\tau \in G_2}\tau(P). $$
Then $f, g \in \mathcal{L}(D)$ and the sublinear system of $|D|$ corresponding to a linear space $\langle f, g, 1\rangle$ is base-point-free. 
Since orders $|G_1|$ and $|G_2|$ are coprime, it follows that $G_1 \cap G_2 =\{1\}$. 
By this condition, the induced morphism
$$ \varphi: X \rightarrow \mathbb{P}^2; \ (f:g:1)$$
is birational onto its image (see \cite[Proposition 1]{fukasawa3}), and of degree $|G_2|$. 
Then the projection from the point $\varphi(P)=(0:1:0)$ coincides with the morphism $(f:1)$ and the projection from the point $Q:=(1:0:0)$ coincides with the morphism $(g:1)$. 
Therefore, $\varphi(P)$ is an inner Galois point with $G_{\varphi(P)}=G_1$ and $Q$ is an outer Galois point with $G_{Q}=G_2$.  
\end{proof}

\section{Proof of Theorem \ref{main2}}
Let $C \subset \mathbb{P}^2$ be an irreducible curve of degree $d \ge 3$, and let $\varphi: X \rightarrow C$ be the normalization. 

\begin{lemma} \label{line}
Let $P \in C$ be an inner Galois point, let $Q \in \mathbb{P}^2 \setminus C$ be an outer Galois point and let $\ell$ be the line passing through $P$ and $Q$. 
Then the following hold. 
\begin{itemize}
\item[(a)] The number of points of ${\rm supp} (\varphi^*\ell)$ is $1$ or $d$. 
\item[(b)] If $G_PG_Q=G_P \rtimes G_Q$ or $G_P \ltimes G_Q$, then $G_P$ fixes $P$. 
In particular, if $\sigma \in G_P \setminus \{1\}$ fixes a point $R \in {\rm supp}(\varphi^*\ell)$, then $R=P$.  
\item[(c)] If $G_PG_Q=G_P \rtimes G_Q$, then $\varphi^*\ell=d P$. 
\end{itemize}
\end{lemma}

\begin{proof}
We consider assertion (a). 
Let 
$$ \varphi^*\ell=\sum_{R \in {\rm supp} (\varphi^*\ell)}n(R) R. $$
Since $Q$ is an outer Galois point and $\ell \ni R$, it follows from \cite[III.7.2]{stichtenoth} that $n(R)=n(R')$ for any $R, R' \in {\rm supp} (\varphi^*\ell)$. 
If $\ell$ is not a tangent line at $P$, then $n(R)=n(P)=1$. 
This implies that the number of points of ${\rm supp} (\varphi^*\ell)$ is equal to $d$. 
Assume that $\ell$ is the tangent line at $P$. 
Then the divisor 
$$\varphi^*\ell-P=(n(P)-1)P+\sum_{R \in {\rm supp} (\varphi^*\ell), \varphi(R) \ne P}n(R)R$$ 
coincides with $(\pi_P \circ \varphi)^*(\pi_P(P))$. 
Assume that $R \in {\rm supp} (\varphi^*\ell)$ with $\varphi(R) \ne P$ exists. 
Since $P$ is a Galois point, it follows that $n(R)=n(P)-1$. 
This is a contradiction. 
Therefore, in this case, $\varphi^*\ell=d P$.  

We consider (b). 
By (a), the number of points of ${\rm supp}(\varphi^*\ell)$ is $1$ or $d$. 
For the former case, it follows from \cite[III.8.2]{stichtenoth} that the assertion holds.   
We can assume that ${\rm supp}(\varphi^*\ell)$ consists of $d$ points. 
Let $\sigma \in G_P \setminus \{1\}$. 
By the assumption, there exist $\sigma' \in G_P \setminus \{1\}$ and $\tau, \tau' \in G_Q \setminus \{1\}$ such that $\tau \sigma=\sigma' \tau'$. 
Then $\tau\sigma(P)=\sigma'\tau'(P) \in \varphi^{-1}(\ell)$. 
Since $\tau^{-1}(\varphi^{-1}(\ell)) \subset \varphi^{-1}(\ell)$, it follows that $\sigma(P) \in T_PC \cap \ell=\{P\}$. 
Therefore, $\sigma(P)=P$ follows.  

We consider (c). 
Let $\tau \in G_Q$ and let $\sigma \in G_P \setminus \{1\}$. 
By (b), $\tau\sigma \tau^{-1}$ fixes $\tau(P)$. 
Since $G_PG_Q=G_P \rtimes G_Q$, $\tau \sigma \tau^{-1} \in G_P \setminus \{1\}$.  
By (b), $\tau(P)=P$. 
It follows from \cite[III.8.2]{stichtenoth} that the ramification index at $P$ for $\pi_Q$ is $d$. 
The assertion follows. 
\end{proof}

\begin{proof}[Proof of Theorem \ref{main2}]
We prove (a) $\Rightarrow$ (b). 
By condition (a) and Lemma \ref{line}, the group $G_P \rtimes G_Q$ fixes $P$. 
It follows from \cite[Lemma 11.44]{hkt} that if $p=0$ or $d(d-1)$ is prime to $p$, then $G_P \rtimes G_Q$ is a cyclic group. 
Therefore, $G_P \rtimes G_Q=G_P \times G_Q$. 
Assume that $p>0$ and $p$ divides $d$. 
Let $d=p^em$, where $m$ is not divisible by $p$. 
It follows from \cite[Lemma 11.44]{hkt} that $G_P \cong \mathbb{Z}/(d-1)\mathbb{Z}$ and there exist subgroups $N, H$ such that $G_Q=N \rtimes H$, $|N|=p^e$ and $H \cong (\mathbb{Z}/m\mathbb{Z})$.
Since $|G_P\rtimes G_Q|=(d-1)p^em$, the subgroup $N$ is a normal subgroup of $G_P \rtimes G_Q$, and hence $G_P \rtimes N=G_P \times N$. 
Since $|G_P \rtimes H|=(d-1)m$ is not divisible by $p$, the group $G_P \rtimes H$ is a cyclic group. 
Therefore, $G_P \rtimes H=G_P\times H$. 
As a consequence, $G_PG_Q=G_P \times G_Q$ follows. 

We prove (b) $\Rightarrow$ (c). 
By assertion (b) in Lemma \ref{line}, $\varphi^*T_PC=d P$. 
By the assumption, $G_Q$ acts on the subfield $k(X)^{G_{P}}$, which is rational. 
Assume that $p=0$ or $d(d-1)$ is prime to $p$. 
Since $G_Q$ is totally ramified at $\pi_P(P)$ and $d$ is prime to $p$, it follows that $G_Q$ is a cyclic group and there exists a ramification point $R \in X/G_{P}$ different from $\pi_P(P)$ (see, for example, \cite[Theorem 11.91]{hkt}). 
The line corresponding to $R$ is denoted by $\ell_Q$. 
By the assumption that $p=0$ or $d-1$ is prime to $p$, a similar argument can be applied to the morphism $X \rightarrow X/G_Q$. 
It is inferred that $G_P$ is a cyclic group and there exists a ramification point $S \in X/G_Q$ different from $\pi_Q(P)$. 
The line corresponding to $S$ is denoted by $\ell_P$. 
The linear system corresponding to $\varphi: X \rightarrow C \subset \mathbb{P}^2$ is determined by three lines $\ell, \ell_Q$ and $\ell_P$. 
For a suitable system of coordinates, we can assume that such lines are defined by $Z=0$, $Y=0$ and $X=0$, respectively. 
Then the covering maps $X \rightarrow X/G_P$ and $X/G_P \rightarrow X/\langle G_P, G_Q \rangle$ are represented by $(X:Y:Z) \mapsto (Y:Z)$ and $(Y:Z) \mapsto (Y^d:Z^d)$. 
When we represent the covering map $X \rightarrow X/G_Q$ by $(X:Y:Z) \mapsto (X:Z)$, we infer that the covering map $X/G_Q \rightarrow X/\langle G_Q, G_P \rangle$ is represented by $(X:Z) \mapsto (aX^{d-1}+bZ^{d-1}:Z^{d-1})$ for some $a, b \in k$, by considering the branch points of the covering map $X \rightarrow X/\langle G_P, G_Q \rangle$. 
Therefore, $y^d=ax^{d-1}+b$ holds in $k(X)$. 
This gives the defining equation of $C$. 
For a suitable system of coordinates, we can assume that $a=-1$ and $b \in \{-1, 0\}$. 

Assume that $p>0$ and $d$ is divisible by $p$. 
Let $d=p^em$, where $m$ is not divisible by $p$.
Since $G_Q$ is totally ramified at $\pi_P(P)$, it follows that $m$ divides $p^e-1$ and there exist subgroups $N, H$ such that $G_Q=N \rtimes H$, $N \cong (\mathbb{Z}/p\mathbb{Z})^{\oplus e}$ and $H \cong \mathbb{Z}/m\mathbb{Z}$ (see, for example, \cite[Theorem 11.91]{hkt}). 
Then there exists a ramification point $R \in X/G_P$ different from $\pi_P(P)$. 
The line corresponding to $R$ is denoted by $\ell_Q$. 
Considering the assumption that $d-1$ is prime to $p$, we infer that $G_P$ is a cyclic group and there exists a ramification point $S \in X/G_Q$ different from $\pi_Q(P)$. 
The line corresponding to $S$ is denoted by $\ell_P$. 
The linear system corresponding to $\varphi: X \rightarrow C \subset \mathbb{P}^2$ is determined by three lines $\ell, \ell_Q$ and $\ell_P$. 
For a suitable system of coordinates, we can assume that such lines are defined by $Z=0$, $Y=0$ and $X=0$, respectively. 
According to Lemma \ref{p^em} below, the covering maps $X \rightarrow X/G_P$ and $X/G_P \rightarrow X/\langle G_P, G_Q \rangle$ are represented by $(X:Y:Z) \mapsto (Y:Z)$ and $(Y:Z) \mapsto (f(Y/Z)Z^d:Z^d)$ for a polynomial 
$$f(y)=\left(\sum_{m \ | \ p^i-1}\alpha_i y^{p^i}\right)^m, $$
where $\alpha_e \ldots, \alpha_0 \in k$ and $\alpha_e \alpha_0 \ne 0$.  
When we represent the covering map $X \rightarrow X/G_Q$ by $(X:Y:Z) \mapsto (X:Z)$, we infer that the covering map $X/G_Q \rightarrow X/\langle G_Q, G_P \rangle$ is represented by $(X:Z) \mapsto (aX^{d-1}+bZ^{d-1}:Z^{d-1})$ for some $a, b \in k$, by considering the branch points of the covering map $X \rightarrow X/\langle G_P, G_Q \rangle$. 
Therefore, $(\sum_{m \ | \ p^i-1} \alpha_i y^{p^i})^m=ax^{d-1}+b$ holds in $k(X)$. 
This gives the defining equation of $C$. 
For a suitable system of coordinates, we can assume that $a=-1$ and $b \in \{-1, 0\}$. 
\end{proof}

\begin{lemma} \label{p^em}
Let $p>0$, $e>0$, $m \ | \ p^e-1$ and let $\psi: \mathbb{P}^1 \rightarrow \mathbb{P}^1$ be a Galois covering of degree $d$ such that the Galois group $G$ is described as $G=N \rtimes H$ for a $p$-subgroup $N$ of order $p^e$ and a cyclic group $H$ of order $m$. 
Assume that $\psi(1:0)=(1:0)$ and $\psi(0:1)=(0:1)$. 
If G fixes $(1:0)$ and $H$ fixes $(0:1)$, then $\psi(Y:Z)=(f(Y/Z)Y^d:Z^d)$ for some polynomial
$$ f(y)=c \left(\sum_{m \ | \ p^i-1}\alpha_iy^{p^i}\right)^m, $$
where $\alpha_e, \ldots, \alpha_0, c \in k$ with $\alpha_e\alpha_0 \ne 0$.  
\end{lemma}

\begin{proof}
The covering $\psi$ is the composite map of the Galois coverings $\psi_1:\mathbb{P}^1 \rightarrow \mathbb{P}^1/N$ and $\psi_2: \mathbb{P}^1/N \rightarrow \mathbb{P}^1/(G/N)$, where $G/N \cong H$. 
We can assume that $\psi_1(1:0)=(1:0)$ and $\psi_1(0:1)=(0:1)$. 
Since $N$ fixes $(1:0)$ and is of degree $p^e$, $\psi_1=(g(Y/Z)Z^{p^e}:Z^{p^e})$ for some polynomial $g(y)$ and there exists an additive subgroup $S \subset k$ such that $N=\{(Y:Z) \mapsto (Y+\alpha Z:Z) \ | \ \alpha \in S\}$. 
It follows from the assumption $\psi_1(0:1)=(0:1)$ that $g(0)=0$.  
Since $g(y)$ is invariant under the action of $N$, it follows from \cite[Proposition 1.1.5 and Theorem 1.2.1]{goss} that 
$$ g(y)=\alpha_e y^{p^e}+\cdots+\alpha_1 y^{p}+\alpha_0 y$$
for some $\alpha_e, \ldots, \alpha_0 \in k$ with $\alpha_e \alpha_0 \ne 0$.  
Since the cyclic covering $\psi_2$ is ramified at points $(1:0)$ and $(0:1)$, $\psi_2(Y:Z)=(c Y^m:Z^m)$ for some $c \in k$. 
Then $\psi=(c(g(y))^m:1)$. 
Let $\zeta$ be a primitive $m$-th root of unity. 
Since $m \ | \ p^e-1$, it follows that $g(\zeta y)=\zeta g(y)$. 
Then $\alpha_i \ne 0$ implies that $(\zeta)^{p^i}=\zeta$. 
We obtain the required form.  
\end{proof}

\begin{remark}
The curves defined by $x^{d-1}+y^d+1=0$ or $x^{d-1}+(\sum_{m \ | \ p^i-1}\alpha_iy^{p^i})^m+1=0$ are smooth. 
For these curves, points $(1:0:0)$ and $(0:1:0)$ are all Galois points. 
This was proved in \cite[Theorem 3]{fukasawa2}. 

For the curve $x^{d-1}+y^d=0$, the same holds. 
It is not difficult to prove it, since $(0:0:1)$ is a cusp and $(1:0:0)$ is a unique total inflection point. 
\end{remark}

\begin{remark} 
\begin{itemize}
\item[(1)]
When $p>0$ and $d-1$ is a power of $p$, $G_PG_Q=G_P \rtimes G_Q$ does not imply $G_PG_Q=G_P \times G_Q$. 
For (the plane model of) the Giulietti--Korchm\'{a}ros curve
$$ x^{q^3}+x-(x^q+x)^{q^2-q+1}-y^{q^3+1}=0 $$ 
(see \cite{giulietti-korchmaros}), where $q$ is a power of $p$, $P=(1:0:0)$ is an inner Galois point and $Q=(0:1:0)$ is an outer Galois point (\cite{fukasawa-higashine}). 
However, $G_Q$ is not a normal subgroup of $G_PG_Q=G_P \rtimes G_Q$. 
\item[(2)] Assume that $p>0$ and $p$ divides $d-1$. 
In this case also, we can prove that $G_Q \cong \mathbb{Z}/d\mathbb{Z}$ and the defining equation of $C$ is of the form $y^d+h(x)=0$, where $h(x)$ is a polynomial of degree at most $d$, up to a projective equivalence. 
\end{itemize} 
\end{remark}

\section{Proof of Theorem \ref{main3}} 
In this section, we consider plane curves admitting an inner Galois point $P$ and an outer Galois point $Q$ with $G_PG_Q=G_P\ltimes G_Q$ and $G_PG_Q \ne G_P \times G_Q$. 

\begin{lemma}
If $d(d-1)$ is prime to $p$, then the number of points of ${\rm supp} (\varphi^*\ell)$ is $d$. 
\end{lemma}

\begin{proof}
By Lemma \ref{line} (a), the number of points of ${\rm supp}(\varphi^*\ell)$ is $1$ or $d$. 
Assume the former. 
Note that $G_P\ltimes G_Q$ fixes $P$. 
It follows from the assumption on the degree and \cite[Lemma 11.44]{hkt} that $G_P \ltimes G_Q$ is a cyclic group.  
This implies that $G_PG_Q=G_P \times G_Q$. 
This is a contradiction.
\end{proof}

Hereafter, we assume that ${\rm supp} (\varphi^*\ell)$ consists of $d$ points. 

\begin{lemma} \label{non-product}
For any $\sigma \in G_P \setminus \{1\}$ and $\tau \in G_Q \setminus \{1\}$, $\sigma\tau \ne \tau\sigma$.     
\end{lemma}

\begin{proof}
Assume that $\sigma\tau=\tau\sigma$. 
By Lemma \ref{line} (b), it follows that
$$ \sigma\tau(P)=\tau\sigma(P)=\tau(P). $$
If $\tau(P) \ne P$, then by Lemma \ref{line} (b), $\sigma(\tau(P)) \ne \tau(P)$. 
Therefore, $\tau(P)=P$. 
Since ${\rm supp}(\varphi^*\ell)$ consists of $d$ points, it follows that $\tau=1$. 
This is a contradiction. 
\end{proof}

We assume that $d-1$ is prime to $p$. 
In this case, the covering map $X/G_Q \rightarrow X/G_PG_Q$ is tame and totally ramified at $\pi_Q(P)$. 
Similarly to the proof of Theorem \ref{main2}, it follows that $G_P$ has an iclusion $G_P \hookrightarrow PGL(3, k)$, $G_P$ is a cyclic group of order $d-1$, and there exists a line $\ell_P \subset \mathbb{P}^2 \setminus \{P\}$ with $\ell_P \ni Q$ such that 
$$ \ell_P \cup \{P\} =\{R \in \mathbb{P}^2 \ | \ \sigma(R)=R \ \mbox{ for all } \sigma \in G_P \}. $$  
For a suitable system of coordinates, $P=(0:1:0)$, $Q=(1:0:0)$, lines $\ell_P$ and $T_PC$ are defined by $Y=0$ and $X=0$ respectively, and $C$ is defined by 
$$ y^{d-1}x+g(x)=0, $$
where $g(x)$ is a polynomial of degree $d$. 

\begin{lemma} \label{non-product2}
The number of points of $(C \setminus {\rm Sing}(C))\cap \ell_P$ is zero or one. 
\end{lemma}

\begin{proof}
Assume that there exist two different smooth points $R_1$ and $R_2 \in \ell_P$. 
Then there exists $\tau \in G_Q \setminus \{1\}$ such that $\tau(R_1)=R_2$. 
Let $\sigma \in G_P \setminus \{1\}$. 
Then 
$$ \tau^{-1}\sigma^{-1}\tau\sigma(R_1)=\tau^{-1}\sigma^{-1}\tau(R_1)=\tau^{-1}\sigma^{-1}(R_2)=\tau^{-1}(R_2)=R_1. $$
Note that $\sigma^{-1}\tau\sigma \in G_Q$, since $G_PG_Q=G_P \ltimes G_Q$. 
Since $\pi_Q$ is not ramified at $R_1$, $\tau^{-1}\sigma^{-1}\tau\sigma=1$. 
Therefore, $\tau\sigma=\sigma\tau$. 
This is a contradiction to Lemma \ref{non-product}.
\end{proof}

\begin{lemma} \label{non-product3}
Assume that $d \ge 4$ and there exists a ramification point $S \in C \setminus \ell_P$ of index $m \ge 2$ for $\pi_Q$. 
Let $n$ be the number of points of $C \cap \ell_P$. 
Then the following holds.  
\begin{itemize}
\item[(a)] If $d \ge 5$, then $n \ge \frac{m-1}{m}d$. 
\item[(b)] $m=2$. 
\item[(c)] There does not exist a point $R \in C \cap \ell_P$ such that the multiplicity of $R$ is two. 
\end{itemize}
\end{lemma}

\begin{proof}
We consider (a). 
On the line $\overline{QS}$, there exist $d/m$ ramification points of index $m$. 
Since any automorphism in $G_P$ is a linear transformation, fixes $Q$ and does not fix any point of $C \cap \overline{QS}$, it follows that there exist $(d-1) \times (d/m)$ ramification points of index $m$. 
It follows from the Riemann--Hurwitz formula that for $\pi_Q$, 
$$ 2g-2 \ge -2d+(d-1) \times \frac{d}{m}(m-1). $$
For $\pi_P$, 
$$ 2g-2 \le -2(d-1)+(d-2)+(d-2)n. $$
With these inequalities used, it follows that 
$$ n \ge \frac{m-1}{m}d-\frac{1}{m}\left(1+\frac{2}{d-2}\right). $$
If $d \ge 5$, then $\frac{1}{m}\left(1+\frac{2}{d-2}\right)<1$. 
Therefore, 
$$ n \ge \frac{m-1}{m}d. $$

We consider (b). 
Assume that $d \ge 5$ and $m \ge 3$. 
Since the number of singular points on $\ell_P$ is at most $d/2$, by (a), the number of smooth points on $\ell_P$ is at least 
$$ \frac{m-1}{m}d-\frac{d}{2}.  $$
If $d \ge 7$ or $m \ge 4$, then this value is strictly more than $1$. 
This is a contradiction to Lemma \ref{non-product2}. 
Assume that $d=6$ and $m=3$. 
In this case, $n \ge 4$ and there exists a smooth point on $\ell_P$. 
Therefore, the number of singular points on $\ell_P$ is at most two, and hence, 
there exist two smooth points on $\ell_P$. 
This is a contradiction to Lemma \ref{non-product2}. 
Assume that $d=4$. 
Then $m=2$ or $4$. 
Assume that $m=4$. 
Then by the proof of (a), 
$$ n \ge \frac{3}{4}d-\frac{1}{4}\left(1+\frac{2}{d-2}\right)=\frac{5}{2}. $$
Therefore, $n \ge 3$. 
Then there exist two smooth points on $\ell_P$. 
This is a contradiction to Lemma \ref{non-product2}. 

We consider (c). 
By (b), $d$ is even. 
Assume that $R \in \ell_P$ is a point of multiplicity two. 
Note that $\overline{RP} \cap C=\{R, P\}$, by considering the defining equation. 
If $\varphi^{-1}(R)$ consists of two points, then the degree of $\pi_P$ is divisible by two. 
Since $d$ is even and $d-1$ is odd, this is a contradiction. 
Therefore, $\varphi^{-1}(R)$ consists of a unique point. 
Such a point also is denoted by $R$. 
Then the ramification index of $R$ for $\pi_P$ is $d-1$, and that of $R$ for $\pi_Q$ is two. 
The stabilizer subgroup $G_{Q}(R)$ for $G_Q$ consists of a unique involution $\tau$. 
Let $\sigma \in G_P \setminus \{1\}$. 
Since $\tau$ and $\sigma$ fix $R$, $\sigma^{-1}\tau\sigma$ fixes $R$. 
It follows from the condition $G_PG_Q=G_P \ltimes G_Q$ that $\sigma^{-1}\tau\sigma \in G_Q(R)$. 
Since $|G_Q(R)|=2$, $\sigma^{-1}\tau\sigma=\tau$. 
This is a contradiction to Lemma \ref{non-product}. 
\end{proof}

\begin{lemma} \label{non-product4} 
Assume that all ramification points for $\pi_Q$ are contained in the line $\ell_P$. 
Then the following holds. 
\begin{itemize}
\item[(a)] $p>0$. 
\item[(b)] The order of each element of $G_Q \setminus \{1\}$ is equal to $p$. 
In particular, $d=p^e$ for some $e>0$. 
\end{itemize} 
\end{lemma}

\begin{proof}
It follows from the Riemann--Hurwitz formula that there exist a wild ramification. 
Therefore, $p>0$. 
In this case, there exists an element $\tau \in G_Q$ of order $p$. 
By Lemma \ref{non-product}, it follows that
$$ G_Q=\{ \sigma^{-1}\tau\sigma \ | \ \sigma \in G_P\} \cup \{1\}. $$
It follows from Sylow's theorem that $d=p^e$ for some $e>0$. 
\end{proof}

\begin{proof}[Proof of Theorem \ref{main3}] 
We consider the only-if part. 
By Lemma \ref{non-product4}, there exists a ramification point $S \in C \setminus \ell_P$ for $\pi_Q$.  
We prove that $d \le 4$. 
Assume that $d \ge 5$. 
By Lemma \ref{non-product3} (a) and (b), $n \ge d/2$. 
The number of smooth points on $\ell_P$ is zero or one, by Lemma \ref{non-product2}. 
If the number is zero, then, by Lemma \ref{non-product3} (c), there exist $d/2$ points of multiplicity at least three on $\ell_P$. 
This is obviously a contradiction. 
Therefore, the number of smooth points on $\ell_P$ is one. 
By Lemma \ref{non-product3} (c), there exist $(d/2)-1$ points of multiplicity at least three on $\ell_P$. 
Then we have an inequality 
$$ 1+3\times \left(\frac{d}{2}-1\right) \le d. $$
This implies that $d \le 4$. 

Let $d=4$. 
By the assumption on the degree, it follows that $p \ne 2, 3$. 
By Lemmas \ref{non-product2} and \ref{non-product3}, there exists a singular point of multiplicity $3$ on $\ell_P$. 
In particular, $C$ is rational. 
We can assume that for a suitable system of coordinates, $P=(0:1:0)$, $Q=(1:0:0)$ and $C$ is defined by 
$$ y^3x+(x+1)^3(x+a)=0, $$
where $a \in k \setminus \{0\}$. 
Let $y=\beta$ be a branch point of $\pi_Q$, where $\beta \in k \setminus \{0\}$. 
Then there exist $c, d \in k$ such that 
$$ \beta^3x+(x+1)^3(x+a)=(x^2+c x+d)^2. $$
Comparing the coefficients of both sides, we have relations 
$$ a+3=2c, \ 3a+3=c^2+2d, \ 3a+\beta^3+1=2c d, \ a=d^2. $$ 
By these equations, we have 
$$ d^4-6d^2+8d-3=(d-1)^3(d+3)=0. $$
If $d=1$, then $\beta=0$. 
This is a contradiction. 
It follows that $d=-3$, $a=9$, $c=6$ and $\beta^3=-64$. 
Then the curve $C$ is defined by 
$$ y^3x+(x+1)^3(x+9)=y^3x+64x+(x^2+6x-3)^2=0. $$

Let $d=3$. 
By the assumption on the degree, it follows that $p \ne 2, 3$. 
By Lemma \ref{non-product2}, there exists a singular point on $\ell_P$. 
Then $C$ is rational. 
Note that the index of all ramification points contained in $C \setminus \ell_P$ for $\pi_Q$ is three. 
We can assume that for a suitable system of coordinates, $P=(0:1:0)$, $Q=(1:0:0)$ and $C$ is defined by 
$$ y^2x+(x+1)^2(x+a)=0, $$
where $a \in k \setminus \{0\}$. 
Let $y=\beta$ be a branch point of $\pi_Q$, where $\beta \in k \setminus \{0\}$. 
Then there exists $c \in k$ such that 
$$ \beta^2x+(x+1)^2(x+a)=(x+c)^3. $$
Comparing the coefficients of both sides, we have relations
$$ a+2=3c, \ 2a+\beta^2+1=3c^2, \ a=c^3. $$
By these equations, we have 
$$ c^3-3c+2=(c-1)^2(c+2)=0. $$
If $c=1$, then $\beta=0$. 
This is a contradiction. 
It follows that $c=-2$, $a=-8$ and $\beta^2=27$. 
Then the curve $C$ is defined by 
$$ y^2x+(x+1)^2(x-8)=y^2x-27x+(x-2)^3=0. $$

To prove the if-part, using Theorem \ref{main1}, we prove the existence of plane rational curves of $d=3$ or $4$ with required properties.
We consider subgroups of ${\rm Aut}(\mathbb{P}^1) \cong PGL(2, k)$. 
It is well known that $A_4$ is a subgroup of $PGL(2, k)$ (see, for example, \cite[Theorem 11.91]{hkt}). 
It is also well known that there exist subgroups $G_1, G_2$ of $A_4$ such that $G_1$ is a cyclic group of order three, $G_2$ is a normal subgroup of $A_4$ with $G_2 \cong (\mathbb{Z}/2\mathbb{Z})^{\oplus 2}$, and $G_1G_2=G_1 \ltimes G_2= A_4$. 
Let $P \in \mathbb{P}^1$ be a point such that $G_1$ fixes $P$. 
We prove that conditions (a) and (b) in Theorem \ref{main1} are satisfied for the triple $(G_1, G_2, P)$. 
By L\"{u}roth's theorem, condition (a) is satisfied. 
With the action of $A_4$ considered, there exist three points $P_2, P_3, P_4$ such that $A_4P=\{P, P_2, P_3, P_4\}$. 
Let $\eta \in G_2 \setminus \{1\}$. 
Then,
$$ P+\sum_{\sigma \in G_1}\sigma(\eta(P))=P+P_2+P_3+P_4=\sum_{\tau \in G_2}\tau(P). $$
This implies condition (b) in Theorem \ref{main1}. 
By Theorem \ref{main1}, the required rational plane curve of $d=4$ is obtained. 
For $d=3$, a similar argument can be applied. 
\end{proof}

\begin{proof}[Proof of Proposition \ref{main4}]
We consider the only-if part. 
We prove that all ramification points for $\pi_Q$ are contained in $\ell_P$.
When $d \ge 5$, according to the first paragraph of the proof of the only-if part of Theorem \ref{main3}, the claim follows. 
Assume by contradiction that there exists a ramification point in $S \setminus \ell_P$ for $\pi_Q$.  
By Lemmas \ref{non-product2} and \ref{non-product3}, if $d \le 4$, then $C$ is rational. 
If $d=4$, then $p=2$ and $\pi_Q$ is wildly ramified at $6$ ramification points contained in $C \setminus \ell_P$, by Lemma \ref{non-product3} (b). 
According to the Riemann--Hurwitz formula for $\pi_Q$, this is a contradiction. 
If $d=3$, then $p=3$ and $\pi_Q$ is wildly ramified at $2$ ramification points. 
This is a contradiction. 
Therefore, the claim follows for all $d \ge 3$. 

By Lemma \ref{non-product4} (a) and (b), $p>0$ and $d=p^e$ for some integer $e>0$.  
Let $R \in \varphi^{-1}(\ell_P)$ and let $G_Q(R)$ be the stabilizer subgroup of $G_Q$ of order $p^k$. 
By the assumption, $k >0$. 
Let $G_P(R)$ be the stabilizer subgroup of $G_P$.  
For any $\sigma \in G_P(R)$, $\sigma^{-1} G_Q(R) \sigma=G_Q(R)$. 
It follows from Lemma \ref{non-product} that $|G_P(R)|+1 \le |G_Q(R)|$. 
With the action of $G_P$ on $\varphi^{-1}(\ell_P)$ considered, the set $\varphi^{-1}(\ell_P)$ contains 
$$ \frac{|G_P|}{|G_P(R)|} \ge \frac{p^e-1}{p^k-1} $$
points. 
With the action of $G_Q$ considered, it follows that 
$$ p^e=|G_Q| \ge |G_Q(R)| \times \frac{p^e-1}{p^k-1}=p^k \times \frac{p^e-1}{p^k-1}. $$
It follows that 
$$ p^{e-k} \le 1. $$
Therefore, $k=e$, that is, $G_Q(R)=G_Q$. 
This implies that $C \cap \ell_P=\{\varphi(R)\}$. 
For a suitable system of coordinates, we can assume that $\varphi(R)=(-1:0:1)$. 
Then $C$ is defined by 
$$ y^{d-1}x+(x+1)^d=0. $$

We consider the if-part. 
The curve $C$ is projectively equivalent to the curve defined by 
$$ x-y^q=0, $$
where $p>0$ and $q$ is a power of $p$. 
Let $P=(0:0:1)$ and let $Q \in \{Z=0\}\setminus \{(1:0:0), (0:1:0)\}$. 
According to \cite{fukasawa-hasegawa}, $P$ is an inner Galois point and $Q$ is an outer Galois point. 
Note that there exists an inclusion $G_P \hookrightarrow PGL(3,k)$ and $G_{P}$ fixes $Q$. 
It is inferred that $G_PG_Q=G_P \ltimes G_Q$. 
Since the line passing through $P$ and $Q$ intersects $C$ at $q$ points, it follows that $G_PG_Q \ne G_P \rtimes G_Q$, by Lemma \ref{line}. 
\end{proof}

\begin{remark}
For the curve of degree $4$ in Theorem \ref{main3}, the point $P$ is a unique inner Galois point. 
This is proved easily, because two flexes are needed for each inner Galois points and there exists $G_P \hookrightarrow PGL(3, k)$.  
For the curve described in Proposition \ref{main4}, the set of Galois points is determined in \cite{fukasawa-hasegawa}.   
\end{remark} 

\begin{remark}
According to a classification of finite subgroups $G$ of $PGL(2, k)$ (see, for example, \cite[Theorem 11.91]{hkt}), a finite non-abelian subgroup $G$ containing subgroups $H_1, H_2 \subset G$ such that $G=H_1 \ltimes H_2$, $|H_2|=|H_1|+1 \ge 4$ and $G$ has a short orbit of length $|H_2|$ must be one of the following: 
\begin{itemize}
\item[(a)] $A_4$ with $p \ne 2, 3$;  
\item[(b)] the semidirect product of an elementary abelian $p$-group of order $q$ with a cyclic group of order $q-1$.  
\end{itemize} 
The cases (a) and (b) are described in Theorem \ref{main3} (b) and Proposition \ref{main4}. 
\end{remark}

\begin{remark}
If ${\rm supp}(\varphi^*\ell)$ consists of $d$ points, then the action of the group $G_P \ltimes G_Q$ on ${\rm supp}(\varphi^*\ell)$ is doubly transitive. 
\end{remark}

\begin{remark}
Assume that ${\rm supp}(\varphi^*\ell)$ consists of $d$ points. 
Let $P_1=P$, let $P_2=\tau(P)$ and let $G_2=\tau G_P \tau^{-1}$, where $\tau \in G_Q \setminus \{1\}$.  
It is inferred that conditions in \cite[Theorem 1]{fukasawa3} are satisfied for the $4$-tuple $(G_1, G_2, P_1, P_2)$. 
Therefore, $X$ has a birational embedding with two inner Galois points. 
Note that $G_1$ fixes $P_1$ and $G_2$ fixes $P_2$. 
In this case, the group $G=\langle G_{P_1}, G_{P_2} \rangle$ ($=G_P \ltimes G_Q$) has been classified by Korchm\'{a}ros, Lia, and Timpanella \cite{klt}. 
It is noted that this birational embedding  is different from the one given by the triple $(G_P, G_Q, P)$ in the case of Theorem \ref{main3} (b). 
In fact, the embedding of $d=4$ described in Theorem \ref{main3} does not admit two inner Galois points.   
\end{remark}

\begin{remark}
Assume that $p>0$ and $p$ divides $d-1$. 
In this case also, we can prove that the defining equation of $C$ is of the form $(\sum_{m \ | \ p^i-1} \alpha_i x^{p^i})^m+h(y)=0$, where $d-1=p^em$, $m$ is not divisible by $p$, $\alpha_e\alpha_0 \ne 0$, and $h(y)$ is a polynomial of degree $d$, up to a projective equivalence. 
\end{remark}

\begin{center} {\bf Acknowledgements} \end{center} 
The author is grateful to Doctor Kazuki Higashine for helpful discussions.

\end{document}